\documentclass{article}
\usepackage{ae}
\usepackage{bm}
\usepackage{graphicx}
\usepackage{ae}
\usepackage{amsmath}
\usepackage{amssymb}
\usepackage{fullpage}
\usepackage{cite}
\usepackage{lineno}
\usepackage{mathrsfs}
\usepackage{amsfonts}
\usepackage{amsthm}

\newtheorem{theorem}{Theorem}[section]
\newtheorem{proposition}[theorem]{Proposition}
\newtheorem{corollary}[theorem]{Corollary}
\newtheorem{lemma}[theorem]{Lemma}

\newtheorem{rel}{Relation}
\newtheorem{prob}{Problem}
\newtheorem{rem}{Remark}[section]
\numberwithin{equation}{section}
\begin{document}
\title{Laplacian spectral characterization of some graph products}

\author{Suijie Wang,\quad Xiaogang Liu
\\
\footnotesize{Department of Mathematics, The Hong Kong University of Science and Technology, Clear Water Bay, Kowloon, Hong Kong, China}
\\
\footnotesize{Email: wangsuijie@gmail.com, lxg.666@163.com}
}

\date{\today}
\maketitle\linenumbers

\begin{abstract}
This paper studies the Laplacian spectral characterization of some graph products. We consider a class of connected graphs: $\mathscr{G}=\left\{G : |EG|\leq|VG|+1\right\}$, and characterize all graphs $G\in\mathscr{G}$ such that the products $G\times K_m$ are $L$-DS graphs. The main result of this paper states that, if $G\in\mathscr{G}$, except for $C_{6}$ and $\Theta_{3,2,5}$, is $L$-DS graph, so is the product $G\times K_{m}$. In addition, the $L$-cospectral graphs with $C_{6}\times K_{m}$ and $\Theta_{3,2,5}\times K_{m}$ have been found.

\bigskip
\noindent{\bf Keywords:} Laplacian Spectrum; $L$-cospectral graphs; $L$-DS graph

\bigskip

\noindent\textbf{AMS classification:} 05C50

\end{abstract}

\section{Introduction}\label{0000}
We start with some basic conceptions of graphs followed from \cite{Biggs}. Let $G=(VG,EG)$ be a graph with vertex set $VG$ and edge set $EG$, where $EG$ is a collection of 2-subsets of $VG$. All graphs considered here are simple and undirected. The \emph{adjacency matrix} $A(G)=(a_{u,v})~(u,v\in VG)$ of $G$ is a matrix whose rows and columns are labeled by $VG$, with $a_{u,v}=1$ if $\{u,v\}\in EG$ and $a_{u,v}=0$ otherwise. The matrix $L(G)=D(G)-A(G)$ is called the \emph{Laplacian matrix} of $G$, where $D(G)$ is a diagonal matrix whose diagonal entry is the degree of the corresponding vertex. Since the matrix $L(G)$ is real and symmetric, its eigenvalues are real numbers and called the \emph{Laplacian eigenvalues} of $G$. It can be shown that $L(G)$ is positive semidefinite. Assuming that $\lambda_1\geq\lambda_2\geq\cdots\geq\lambda_n(=0)$ are these eigenvalues, the multiset $\mathrm{Spec}(G)=\{\lambda_1,\ldots,\lambda_n\}$ is called the \emph{Laplacian spectrum} of $G$. For simplicity, we write $[\lambda_i]^{m_i}\in\mathrm{Spec}(G)$ to denote that the multiplicity of $\lambda_i$ is $m_i$. Two graphs are said to be  \emph{$L$-cospectral} if they share the same Laplacian spectrum. Two graphs $G$ and $H$ are said to be \emph{isomorphic} if there is a bijection between $VG$ and $VH$ which induces a bijection between $EG$ and $EH$. Throughout this paper, we write $G=H$ whenever $G$ and $H$ are isomorphic. A graph $G$ is called to be \emph{determined by its Laplacian spectrum}, or  \emph{$L$-DS graph} for short, if all graphs $L$-cospectral with $G$ are isomorphic to $G$.

Given two graphs $G_1$ and $G_2$ with disjoint vertex sets $VG_1$ and $VG_2$ and edge sets $EG_1$ and $EG_2$, the \emph{disjoint union}, or \emph{addition} for convenience, of $G_1$ and $G_2$ is defined to be the graph $G=(VG_1\cup VG_2, EG_1\cup EG_2)$, denoted by $G_1+ G_2$. Especially, $\underbrace{G+\cdots+G}_m$ is denoted by $mG$. The \emph{product} of graphs $G_1$ and $G_2$ is the graph $G_1+G_2$ together with all the edges joining $VG_1$ and $VG_2$, denoted by $G_1\times G_2$. Let $K_m$ be the \emph{complete graph} of $m$ vertices, $P_{m}$ the \emph{path} of $m$ vertices, and $C_m$ the \emph{cycle} of $m$ vertices, respectively. Clearly, the complete graph $K_m$ can be written as the product of $m$ isolated vertices. Let $K_1$ be an isolated vertex, then $K_m=\underbrace{K_1\times\cdots\times K_1}_m$. Similarly,  $mK_{1}=\underbrace{K_1+\cdots+ K_1}_m$ denotes the disjoint union of $m$ isolated vertices. A connected graph is called a  \emph{tree} if it contains no cycle, \emph{unicyclic} if exactly one cycle, and \emph{bicyclic} if two independent cycles. Let $G$ be a connected graph. A subgraph $S$ of $G$ is called a  \emph{spanning tree} of $G$ if $S$ is a tree and $VS=VG$. Denote by $s(G)$ the number of spanning trees of $G$. Obviously, $s(G)=0$ if $G$ is disconnected. These notations will be fixed throughout this paper.


This paper is to characterize which graph products are determined by their Laplacian spectra. It is motivated by \cite{kn:Zhang09,kn:Liu07} that we propose the following problem.

\begin{prob}\label{prob1}
Characterize all graphs $G$ such that $G\times K_m$ are $L$-DS graphs.
\end{prob}

In \cite{kn:Zhang09}, the wheel graph $C_n\times K_1$ for $n\neq 6$ is proved to be $L$-DS graph. In the conclusion, the authors posed an interesting question. The question is that which graphs satisfy the following relation:
\begin{rel}\label{rel1}
If $G$ is a $L$-DS graph, then $G\times K_1$ is also a $L$-DS graph.
\end{rel}
Clearly, Relation \ref{rel1} is just a special case of Problem \ref{prob1}. It is known that if $G$ is disconnected, i.e., $G$ has at least two components, then $G$ always satisfies Relation \ref{rel1} (see Proposition 4 in \cite{kn:vandam07}). If $G$ is connected, we know that cycle $C_n$ with $n\neq6$ and path $P_n$ satisfy Relation \ref{rel1} \cite{kn:Zhang09,kn:Liu07}.

\begin{figure}[here] \setlength{\unitlength}{4pt}
\begin{picture}(25,23)(-16,-12)
\put(9,0){\circle*{1}} \put(13,8){\circle*{1}}
\put(13,-8){\circle*{1}} \put(21,8){\circle*{1}}
\put(21,-8){\circle*{1}} \put(25,0){\circle*{1}}
\put(17,-0){\circle*{1}} \put(17,0){\line(1,0){8}}
\put(17,0){\line(-1,0){8}} \put(17,0){\line(1,2){4}}
\put(17,0){\line(1,-2){4}} \put(17,0){\line(-1,2){4}}
\put(17,0){\line(-1,-2){4}} \put(9,0){\line(1,2){4}}
\put(9,0){\line(1,-2){4}} \put(21,8){\line(-1,0){7.5}}
\put(21,8){\line(1,-2){4}} \put(21,-8){\line(-1,0){7.5}}
\put(21,-8){\line(1,2){4}} \put(13,-12){\small{$C_{6} \times K_{1}$}}
\end{picture}
\begin{picture}(30,23)(-45,-12)
\put(-16,0){\circle*{1}} \put(-8,0){\circle*{1}}
\put(0,0){\circle*{1}} \put(8,0){\circle*{1}}
\put(16,0){\circle*{1}} \put(0,8){\circle*{1}}
\put(0,-8){\circle*{1}}\put(-16,0){\line(1,0){8}}
\put(8,0){\line(-1,0){8}} \put(0,8){\line(-2,-1){16}}
\put(0,8){\line(-1,-1){8}} \put(0,8){\line(0,-1){8}}
\put(0,8){\line(1,-1){8}} \put(0,8){\line(2,-1){16}}
\put(0,-8){\line(-2,1){16}} \put(0,-8){\line(-1,1){8}}
\put(0,-8){\line(0,1){8}} \put(0,-8){\line(1,1){8}}
\put(0,-8){\line(2,1){16}}
 \put(-9,-12){\small{$H_1=2K_{1}\times (2P_{2}+ K_{1})$}}
\end{picture}
\caption {The $L$-cospectral graphs $C_{6}\times K_{1}$  and
$H_1$}\label{F1}
\end{figure}
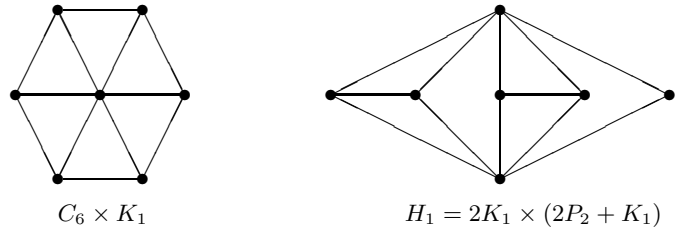

\begin{figure}[here]
\setlength{\unitlength}{4pt}
\begin{picture}(25,23)(-16,-12)
\put(9,0){\circle*{1}} \put(13,8){\circle*{1}}
\put(13,-8){\circle*{1}} \put(21,8){\circle*{1}}
\put(21,8){\line(0,-1){16}} \put(21,-8){\circle*{1}}
\put(25,0){\circle*{1}} \put(17,-0){\circle*{1}}
\put(17,0){\line(1,0){8}} \put(17,0){\line(-1,0){8}}
\put(17,0){\line(1,2){4}} \put(17,0){\line(1,-2){4}}
\put(17,0){\line(-1,2){4}} \put(17,0){\line(-1,-2){4}}
\put(9,0){\line(1,2){4}} \put(9,0){\line(1,-2){4}}
\put(21,8){\line(-1,0){7.5}} \put(21,8){\line(1,-2){4}}
\put(21,-8){\line(-1,0){7.5}} \put(21,-8){\line(1,2){4}}
\put(10.5,-11.5){\small{$\Theta_{3,2,5} \times K_{1}$}}
\end{picture}
\begin{picture}(30,23)(-45,-12)
\put(-16,0){\circle*{1}} \put(-8,0){\circle*{1}}
\put(0,0){\circle*{1}} \put(8,0){\circle*{1}}
\put(16,0){\circle*{1}} \put(0,8){\circle*{1}}
\put(0,-8){\circle*{1}}\put(-16,0){\line(1,0){8}}
\put(8,0){\line(-1,0){16}} \put(0,8){\line(-2,-1){16}}
\put(0,8){\line(-1,-1){8}} \put(0,8){\line(0,-1){8}}
\put(0,8){\line(1,-1){8}} \put(0,8){\line(2,-1){16}}
\put(0,-8){\line(-2,1){16}} \put(0,-8){\line(-1,1){8}}
\put(0,-8){\line(0,1){8}} \put(0,-8){\line(1,1){8}}
\put(0,-8){\line(2,1){16}}
\put(-9,-11.5){\small{$H_2=2K_{1} \times (P_{4}+ K_{1})$}}
\end{picture}
\caption {The $L$-cospectral graphs $\Theta_{3,2,5}\times K_{1}$ and
$H_2$}\label{F2}
\end{figure}
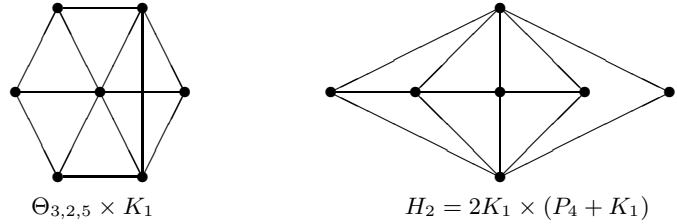

In this paper, we consider a class of connected graphs: $\mathscr{G}=\left\{G : |EG|\leq|VG|+1\right\}$, and characterize all graphs $G$ among $\mathscr{G}$ such that $G\times K_m$ are $L$-DS graphs.  Indeed, $\mathscr{G}$ consists of all connected trees, connected unicyclic graphs and connected bicyclic graphs. To characterize which connected trees satisfy Problem \ref{prob1} are investigated in Section \ref{1111}.  And we show that if a connected tree $T$ is $L$-DS, so is $T\times K_m$. The characterization for unicyclic graphs are investigated in Section \ref{2222}. We prove that if a connected unicyclic graph $U\ne C_6$ is $L$-DS, then $U\times K_m$ is also $L$-DS. At last, we consider the products of $L$-DS bicyclic graphs and $K_m$. It is shown that all $L$-DS bicyclic graphs, except for $\Theta_{3,2,5}$, satisfy Problem \ref{prob1}, where
$\Theta_{3,2,5}$ denotes the graph consisting of two cycles $C_{3}$ and $C_{5}$ who share a common path $P_2=C_3\cap C_5$. Meanwhile we find one new  pair of $L$-cospectral graphs, which are $\Theta_{3,2,5}\times K_{m}$ and $H_2\times K_{m-1}$, see Figure \ref{F2} for the case $m=1$. Indeed, $L$-cospectral graphs shown in Figure \ref{F1}, which are posed in \cite{kn:Zhang09}, can also be figured out by our proof in Section \ref{2222}.

\section{Preliminaries}
In this section, we mention some results, which will be used later.
\begin{lemma}\label{spanning tree}{\rm \cite{Biggs}} Let $\{\lambda_{1},\ldots,\lambda_{n-1},0\}$ be the Laplacian spectrum of the graph $G$. Then
\begin{eqnarray*}
 s(G)=\dfrac{\lambda_{1}\lambda_{2}\cdots \lambda_{n-1}}{n}.
\end{eqnarray*}
\end{lemma}

\begin{lemma}\label{Laplacian spectrum}{\rm\cite{kn:Kelmans74,kn:vandam03}}
Let $G$ be a graph. The following can be determined by its Laplacian spectrum:\vspace{1.5mm}
\\
\emph{(1)} The number of vertices of $G$. \\
\emph{(2)} The number of edges of $G$. \\
\emph{(3)} The number of components of $G$.\\
\emph{(4)} The number of spanning trees of $G$. \\
\emph{(5)} The sum of the squares of degrees of vertices.
\end{lemma}




\begin{lemma}\label{product laplacian}{\rm\cite{kn:Merris98}}
Let $G$ and $H$ be two graphs with $|VG|=n$ and $|VH|=m$. Suppose $\mathrm{Spec(}G)=\{\mu_{1},\mu_{2},\ldots,\mu_{n-1},0\}$ and $\mathrm{Spec}(H)=\{\nu_{1},\nu_{2},\ldots,\nu_{m-1},0\}$. Then the Laplacian spectrum of the product $G\times H$ is
\begin{eqnarray*}
 \mathrm{Spec}(G\times H)=\{n+m,m+\mu_{1},\ldots, m+\mu_{n-1},n+\nu_{1},\ldots, n+\nu_{m-1},0\}.
\end{eqnarray*}
\end{lemma}

\begin{lemma}\label{lemma1}
Suppose $G$ is a $L$-DS graph. If there is a graph $H$ and a positive integer $m$ such that $\mathrm{Spec}(G\times K_m)=\mathrm{Spec}(H\times K_m)$, then we have $G= H$.
\end{lemma}
\begin{proof}
Since $\mathrm{Spec}(G\times K_m)=\mathrm{Spec}(H\times K_m)$, Lemma \ref{product laplacian} implies that $ \mathrm{Spec}(G)=\mathrm{Spec}(H)$. Therefore, $G= H$ since $G$ is a $L$-DS graph.
\end{proof}

\begin{lemma}\label{product graphs}{\rm\cite{Grone}} Let $G$ be a connected graph with $n$ vertices. Then $n$ is the Laplacian eigenvalue with multiplicity $k$ if and only if $G$ is the product of exactly $k+1$ graphs.
\end{lemma}






\begin{lemma}\label{2nd max eigenvalue}{\rm\cite{kn:R.Merris}}
Let $G$ be a graph and $\lambda(G)$ the largest Laplacian eigenvalue of $G$. Denote by $d(v)$ the vertex degree of $v\in VG$. Then
\[\lambda(G)\leq \max\{d(v)+m(v)|v\in VG\},\]
where $m(v)=\dfrac{1}{d(v)}\sum_{\{u,v\}\in EG}d(u)$ is the average of degrees for all neighbors of $v$.
\end{lemma}

\begin{lemma}\label{min}{\rm\cite{kn:Merris94}} Let $\lambda(G)$ and $\Delta(G)$ be the maximum Laplacian eigenvalue and the maximum vertex degree of $G$, respectively. If $G$ has at least one edge, then $\lambda(G)\geq \Delta(G)+1$. Moreover, if $G$ is connected graph of $n$ vertices with $n>1$, then we have
\begin{eqnarray*}
\lambda(G)=\Delta(G)+1\Longleftrightarrow\Delta(G)=n-1.
\end{eqnarray*}
\end{lemma}

\section{Laplacian spectral characterization of the products of  trees and complete graphs}\label{1111}
In this section, the main result states that the products of $L$-DS trees and complete graphs are $L$-DS graphs. To prove this result, we first need one number theoretic proposition.

\begin{proposition}\label{fact1}Let $s$ and $t$ be two positive integers. If $x_{0},x_{1},\ldots, x_{k}$ is a sequence of
integers with $\sum_{i=0}^{k}x_{i}=t$ and $x_{i}\geq s$ for all $i$, then we have
\begin{equation}\label{inequality1}
    \sum_{i=0}^{k}x_{i}^{2}\leq (t-ks)^{2}+ks^2,
\end{equation}
where the equality of {\mbox{\rm(\ref{inequality1})}} holds if and only if all $x_i$ are identically $s$ but one equals to $t-ks$.
\end{proposition}

\begin{proof}
We shall use induction on $k$ to prove this result. It is obvious when $k=0$. For $k\ge 1$, $\sum_{i=0}^{k}x_{i}=t$ implies that $\sum_{i=0}^{k-1}x_{i}=t-x_k$. By the induction hypothesis, we obtain that
\begin{eqnarray}
\sum_{i=0}^{k}x_{i}^{2} &\leq & x_k^2+(t-x_k-(k-1)s)^{2}+(k-1)s^2\label{inequality2}\\
&=& \frac{1}{2}[2x_k-t+(k-1)s]^2+\frac{1}{2}[t-(k-1)s]^2+(k-1)s^2.\label{inequality22}
\end{eqnarray}
Notice that $\sum_{i=0}^{k}x_i=t$ and $x_i\ge s$ for all $i$. Then we have $s\le x_k\le t-ks$. Thus,
\begin{equation}\label{inequality3} -t+(k+1)s\le 2x_k-t+(k-1)s\le t-(k+1)s.\end{equation}
Applying (\ref{inequality3}) to (\ref{inequality22}), we can obtain that
\begin{eqnarray*}
\sum_{i=0}^{k}x_{i}^{2}\le \frac{1}{2}[t-(k+1)s]^2+\frac{1}{2}[t-(k-1)s]^2+(k-1)s^2=(t-ks)^{2}+ks^2.
\end{eqnarray*}
Note that the equality of (\ref{inequality1}) holds if and only if the equalities of both (\ref{inequality2}) and (\ref{inequality3}) hold simultaneously. Clearly, the equality of (\ref{inequality3}) holds if and only if $x_k=s$ or $t-ks$. If $x_k=t-ks$, then it is easy to obtain that $x_i=s$ for $i\le k-1$, since $\sum_{i=0}^{k}x_i=t$ and $x_i\ge s$ for all $i$. Meanwhile, the equality of (\ref{inequality2}) holds in this case. If $x_k=s$, by the induction hypothesis, the equality of (\ref{inequality2}) holds if and only if all $x_i$ are identically $s$ but one for all $i\le k-1$. This completes the proof.
\end{proof}

\begin{lemma}\label{lemma2}
If a tree $T$ is $L$-DS, so is the product $T\times K_{1}$.
\end{lemma}
\begin{proof}
To prove $T\times K_1$ is $L$-DS, assume that $G$ is a graph $L$-cospectral to $T\times K_{1}$. We need to prove that $G$ is isomorphic to $T\times K_1$. If $|VT|=n$, by Lemma \ref{Laplacian spectrum}, $G$ is a connected graph with $|VG|=n+1$. By Lemmas \ref{product laplacian} and \ref{product graphs}, $G$ can be written as the product of two graphs, then say $G=G_1\times G_2$. Fix the following notations,
\begin{eqnarray*}
v_1=|VG_{1}|,~~e_{1}=|EG_{1}|,~~e_{2}=|EG_{2}|.
\end{eqnarray*}
Without loss of generality, we assume $|VG|\ge 2|VG_1|$, i.e., $n+1\geq 2v_1$. Counting the edges of both $G$ and $T\times  K_{1}$ and applying Lemma \ref{Laplacian spectrum}, we obtain $e_1+e_2+v_1(n+1-v_1)=2n-1$. It follows that
\begin{equation}\label{edgessum}
e_{1}+e_{2}=(2-v_1)n+v_1^2-v_1-1.
\end{equation}
From Lemma \ref{lemma1}, we only need to show that $v_1=1$, viz. $G=K_1\times G_2$. Now suppose $v_1\geq 2$. Applying $n+1\geq 2v_1$ and $v_1\ge 2$ to (\ref{edgessum}), we have
\begin{equation}\label{edgesumequality}
e_{1}+e_{2}\leq (2-v_1)(2v_1-1)+v_1^{2}-v_1-1=-(v_1-1)(v_1-3).
\end{equation}
Note that $e_1+e_2\ge 0$. It forces $v_1=2$ or $3$. Then our proof will be complete with the following cases.\vspace{2mm}

\noindent\emph{Case 1}. $v_1=2$. Equation (\ref{edgessum}) implies $e_{1}+e_{2}=1$. Then we have $e_{1}=1$ or $e_{1}=0$.\vspace{2mm}

\noindent\emph{Case 1.1}. $e_{1}=1$. Since $v_1=2$, it is easily seen that $G_1= K_2=K_1\times K_1$.  It follows that $G=G_{1}\times G_2= K_{1}\times (K_{1}\times G_2)$. Since $G$ is $L$-cospectral to $T\times K_1$, applying Lemma \ref{lemma1}, we have $G= T\times K_1$.\vspace{2mm}

\noindent\emph{Case 1.2}. $e_{1}=0$. Applying $v_1=2$ and $e_1+e_2=1$, we can easily obtain that $G_1= 2K_1$ and $G_2=(n-3)K_1+P_2$. Since $G=G_1\times G_2$, by routine calculations, we have  $\mathrm{Spec}(G_2)=\{2,[0]^{n-2}\}$. Applying Lemma \ref{product laplacian}, we have
\begin{eqnarray*}
\mathrm{Spec}(G)=\{n+1,n-1,4,[2]^{n-3},0\}.
\end{eqnarray*}
Since $\mathrm{Spec}(T\times K_1)=\mathrm{Spec}(G)$, by Lemma \ref{product laplacian}, the Laplacian spectrum of $T$ is
\begin{eqnarray*}
 \mathrm{Spec}(T)=\{n-2,3,[1]^{n-3},0\}.
\end{eqnarray*}
By Lemma \ref{spanning tree}, the number of spanning trees of $T$ is given by $s(T)=\frac{3(n-2)}{n}$. But obviously $s(T)=1$. It follows that $n=3$. Hence, $G_2= P_2$, and then $G= 2K_1\times P_2= K_{1}\times  P_{3}$. Now we can complete this case easily by applying Lemma \ref{lemma1}.\vspace{2mm}

\noindent\emph{Case 2}. $v_1=3$. Equation (\ref{edgesumequality}) implies $e_{1}=e_{2}=0$. Applying $v_1=3$ and $e_{1}=e_{2}=0$ to (\ref{edgessum}), we can obtain $n=5$. It follows that $G_1=3K_1$ and $G_2=3K_1$, and then $G=3K_1\times 3K_1$. Its Laplacian spectrum is $\{6,[3]^4,0\}$. Since $\mathrm{Spec}(T\times K_1)=\mathrm{Spec}(G)$, by Lemma \ref{product laplacian}, the Laplacian spectrum of $T$ is
$\{[2]^4,0\}$. Apply Lemma \ref{spanning tree}, we have $s(T)=\frac{16}{5}$, which is a contradiction.
\end{proof}

\begin{theorem}\label{treekm}
If a tree $T$ is $L$-DS, so is the product $T\times K_m$ for all positive integers $m$.
\end{theorem}

\begin{proof}Suppose the graph $G$ is $L$-cospectral to $T\times K_m$. We shall use induction on $m$ to show that $G= T\times K_m$. The case $m=1$ is stated in Lemma \ref{lemma2}. Now we assume $m\ge 2$. Note that
\[T\times K_{m}=T\times \underbrace{K_1\times \cdots\times  K_1}_m.\]
Since $\mathrm{Spec}(G)=\mathrm{Spec}(T\times K_m)$, by Lemma \ref{product graphs}, $G$ is the product of $m+1$ graphs, denoted
\[G=G_{0}\times G_{1}\times \cdots\times G_{m}.\]
Fix notations as follows,
\begin{equation}\label{notation}n=|VT|,~~e_{i}=|EG_{i}|,~~v_{i}=|VG_{i}| \text{~~~~for~~}i=0,1,\dots,m.\end{equation}
Without loss of generality, assume $v_{0}\geq v_{1}\geq\cdots\geq v_{m}$. It is obvious that $\sum_{i=0}^{m}v_i=n+m$ by Lemma \ref{Laplacian spectrum}. In the following, we are going to prove $v_{m}=1$ by contradiction. Now suppose $v_{m}\geq 2$. It follows that $v_i\ge 2$ for all $i=0,\ldots,m$. Then we have $m+n=\sum_{i=0}^{m}v_i\ge 2(m+1)$, so $n\ge m+2$. For convenience, we list those conclusions as follows,
\begin{equation}\label{eq5}
m\ge 2,~~v_0\ge\cdots\ge v_m\ge 2,~~m+n=\sum_{i=0}^{m}v_i,~~n\ge m+2.
\end{equation}
 Combining $v_0\ge\cdots\ge v_m\ge 2$ with $\sum_{i=0}^{m}v_i=n+m$, by Proposition \ref{fact1}, we have
\begin{equation}\label{eq4}
\sum_{i=0}^{m}v_i^2\le (n-m)^2+4m.
\end{equation}
Since $\mathrm{Spec}(G)=\mathrm{Spec}(T\times K_m)$, Lemma \ref{Laplacian spectrum} implies that $G$ and $T\times K_m$ have the same number of edges. Counting the edges of both $G$ and $T\times K_{m}$, we have
\begin{equation}\label{numberedges}
    \sum_{i=0}^{m}e_{i}+\sum_{0\leq i<j\leq m}v_{i}v_{j}=n-1+mn+\frac{m(m-1)}{2}.
\end{equation}
Since $\sum_{i=0}^{m}v_i=n+m$, we have
\begin{equation}\label{sincess}
    \sum_{0\leq i<j\leq m}v_{i}v_{j}=\frac{1}{2}\left(\left(\sum_{i=0}^{m}v_{i}\right)^{2}-\sum_{i=0}^{m}v_{i}^{2}\right)
=\frac{1}{2}\left(\left(n+m\right)^{2}-\sum_{i=0}^{m}v_{i}^{2}\right),
\end{equation}
Applying (\ref{sincess}) to (\ref{numberedges}), we obtain
\begin{eqnarray}\label{sumes}
\sum_{i=0}^{m}e_{i}=\frac{1}{2}\left(\sum_{i=0}^{m}v_{i}^{2}-n^{2}-m\right)+n-1.
\end{eqnarray}
Applying (\ref{eq4}) to (\ref{sumes}), we have
\begin{equation}\label{syezss}
   \sum_{i=0}^{m}e_{i}\leq (1-m)n+\frac{1}{2}(m^2+3m)-1.
\end{equation}
Applying $m\geq 2$ and $n\ge m+2$ of (\ref{eq5}) to (\ref{syezss}), we have
\begin{equation}\label{bianfanwei}
    \sum_{i=0}^{m}e_{i}\leq -\frac{1}{2}(m^2-m-2).
\end{equation}
Notice that $-\frac{1}{2}(m^2-m-2)\le 0$ for $m\ge 2$, but $\sum_{i=0}^{m}e_{i}\ge 0$. It follows that
\begin{equation}\label{eq6}
m=2,~~e_i=0  \text{~~~~for~~}  i=0,1,2.
\end{equation}
Combining (\ref{eq6}), (\ref{syezss}), and $n\ge m+2$ of (\ref{eq5}), we obtain $n=4$. So far, we have obtained that $G=G_0\times G_1\times G_2$ satisfies
\[|VG_0|\ge |VG_1|\ge |VG_2|\ge 2,~~|EG_0|=|EG_1|=|EG_2|=0,~\text{~and~}~|VG|=m+n=6.\]
It follows that
\[G=2K_{1}\times  2K_{1} \times 2K_{1}\]
Then we have $\mathrm{Spec}(G)=\{[6]^2,[4]^3,0\}$. Since $\mathrm{Spec}(G)=\mathrm{Spec}(T\times K_m)$, applying Lemma \ref{product laplacian}, we have $\mathrm{Spec}(T)=\{[2]^3,0\}$. By Lemma \ref{spanning tree}, the number of spanning trees of $T$ is $s(T)=2$. Note the fact that $T$ is a tree. It is a contradiction. Now we have shown that $v_{m}=1$, and then $G_{m}=K_{1}$.  From $\mathrm{Spec}(T\times K_m)=\mathrm{Spec}(G)$, we have
\[\mathrm{Spec}(K_1\times (T\times K_{m-1}))=\mathrm{Spec}(K_1\times (G_0\times \cdots\times G_{m-1})).\]
By Lemma \ref{product laplacian}, we have
\[\mathrm{Spec}(T\times K_{m-1})=\mathrm{Spec}(G_0\times \cdots\times G_{m-1}).\]
By the induction hypothesis of $m-1$,
\[T\times K_{m-1}=  G_{0}\times  \cdots \times G_{m-1}.\]
Thus $T\times K_{m}= G_{0}\times \cdots \times  G_{m}=G$. The proof is complete.
\end{proof}

\begin{rem}
Up until now, there are so many trees are proved to be $L$-DS graphs, for examples, path $P_{n}$ \cite{kn:vandam03}, graphs $Z_n$, $T_n$, and $W_n$ \cite{kn:Shen05}, starlike tree $S$ \cite{kn:Omidi07}, etc. Therefore, Theorem \ref{treekm} implies that $P_{n}\times K_{m}$, $Z_n\times K_m$, $T_n\times K_m$, $W_n\times K_m$ and $S\times K_m$ are also $L$-DS graphs.
\end{rem}

\section{Laplacian spectral characterization of the products of unicyclic graphs and  complete graphs}\label{2222}
This section is devoted to the Laplacian spectral characterization of the products of unicyclic graphs and complete graphs. Recall that a unicyclic graph is a connected graph containing exactly one cycle. In other words, a connected graph $G=(VG,EG)$ is unicyclic iff $|VG|=|EG|$.  In notations, we write the unicyclic graph as $U$. For $k\leq n$, denote by $\mathcal{U}(n,k)$ the collection of all unicyclic graphs $U$ with $|VU|=n$ and containing the cycle $C_k$ as a subgraph. Recall in Lemma \ref{2nd max eigenvalue} that given a vertex $v$ of the graph $G$, $d(v)$ denotes the degree of $v$, and $m(v)$ is defined to be $m(v)=\dfrac{1}{d(v)}\sum_{\{u,v\}\in EG}d(u).$
\begin{proposition}\label{8.1} With above notations, for all $U\in \mathcal{U}(n,k)$, we have
\begin{equation}\label{eq7}
\max\{d(v)+m(v)\mid v\in VU\}\leq n-k+3+\frac{2}{n-k+2}.
\end{equation}
The equality of {\rm (\ref{eq7})} holds if and only if $U$ is the graph obtained by appending $n-k$ vertices to a vertex of the cycle $C_k$.
\end{proposition}
\begin{proof}
Since $U\in \mathcal{U}(n,k)$ contains the cycle $C_k$ as a subgraph, then $|VU\setminus VC_k|=n-k$. It is easily seen that the maximum vertex degree of $U$ is $n-k+2$, viz.
\begin{equation}\label{eq8}
d(v)\le n-k+2\text{~~~~for all~~}  v\in VU.
\end{equation}
Given $v_0\in VU$, we shall prove (\ref{eq7}) by studying the following cases of $d(v_0)$.\vspace{2mm}

\noindent\emph{Case 1}. $d(v_0)=1$.   Clearly, $v_0\notin VC_k$ and there is a unique vertex adjacent to $v_0$, denoted $v\in VU$. By (\ref{eq8}), we have $d(v)\leq n-k+2$. Thus
\[d(v_0)+m(v_0)=d(v_0)+d(v)\leq n-k+3.\]

\noindent\emph{Case 2}. $n-k+2\ge d(v_0)\ge 2$.   To prove (\ref{eq7}), viz. to find the maximum value $m(v_0)$ for $v_0$ with fixed $d(v_0)$, it is enough to find the maximum value of the sum
\begin{equation}\label{sum}
\sum_{\{v,v_0\}\in EG}d(v).
\end{equation}
Note that $U$ is unicyclic with the cycle $C_k$. Consider the following vertex set
\[V_0=\{u\in VU\setminus VC_k\mid u\; \mbox{is not adjacent to}\; v_0 \}.\]
Since $U$ is unicyclic, $v_0$ has at most two neighbors in $VC_k$. And $v_0$ has two neighbors in $VC_k$ occurs only when $v_0\in VC_k$. It implies that
\begin{equation}\label{V_0}
|V_0|\le n-k-d(v_0)+2.
\end{equation}
In order to make the sum (\ref{sum}) as large as possible, assume that all vertices of $V_0$ are adjacent to neighbors of $v_0$. Now, the sum (\ref{sum}) equals \[n-k-d(v_0)+2+d(v_0)+2=n-k+4.\]
Clearly, this is the maximum value for the sum (\ref{sum}).
Thus, in general, we have
\[\sum_{\{v,v_0\}\in EU}d(v)\leq n-k+4.\]
It follows that
\[d(v_0)+m(v_0)\leq d(v_0)+\frac{n-k+4}{d(v_0)}.\]
Now we are going to find an upper bound of $d(v_0)+\frac{n-k+4}{d(v_0)}$ with $2\le d(v_0)\le n-k+2$. Note that the maximum value of $d(v_0)+\frac{n-k+4}{d(v_0)}$ occurs only when $d(v_0)=2$ or $n-k+2$. On the other hand, to compare these two values, we have
\begin{eqnarray*}
\left(n-k+2+\frac{n-k+4}{n-k+2}\right)-\left(2+\frac{n-k+4}{2}\right)=\frac{n-k+2}{2}+\frac{2}{n-k+2}-2\geq 0.
\end{eqnarray*}
It is easily seen that $d(v_0)+\frac{n-k+4}{d(v_0)}$ is maximum iff $d(v_0)=n-k+2$. Hence,
\[d(v_0)+m(v_0)\leq n-k+3+\frac{2}{n-k+2},\]
where the equality holds iff $d(v_0)=n-k+2$. Note that $d(v_0)=n-k+2$ implies that $V_0=\emptyset$ by (\ref{V_0}). This completes the proof.
\end{proof}
\begin{lemma}\label{unicyclic lemma} If a unicyclic graph $U$ is $L$-DS and $U\neq C_{6}$, then $U\times K_{1}$ is $L$-DS. Moreover, $C_{6}\times K_{1}$ is $L$-cospectral to $2K_{1}\times  (2P_{2}+ K_{1})$, see Figure \ref{F1}.
\end{lemma}
\begin{proof} The idea of the proof is almost the same as Lemma \ref{lemma2}. Similarly, assume that $G$ is a graph $L$-cospectral to $U\times K_{1}$. We shall determine the condition, under which $G$ is isomorphic to $U\times K_1$. Let $|VU|=n$, by Lemma \ref{Laplacian spectrum}, then $G$ is a connected graph with $|VG|=n+1$. By Lemmas \ref{product laplacian} and \ref{product graphs}, $G$ can be written as the product of two graphs $G_1$ and $G_2$, i.e., $G=G_1\times G_2$. Fix the following notations,
\[v_1=|VG_1|,~~ e_{1}=|EG_{1}|, ~~ e_{2}=|EG_{2}|.\]
Without loss of generality, we assume $|VG|\ge 2|VG_1|$, i.e., $n+1\geq 2v_1$. Counting the edges of both $G$ and $U\times K_{1}$ and applying Lemma \ref{Laplacian spectrum}, we obtain $e_1+e_2+v_1(n+1-v_1)=2n$. It follows that
\begin{equation}\label{edgessum1}e_{1}+e_{2}=(2-v_1)n+v_1^2-v_1. \end{equation}
From Lemma \ref{lemma1}, it would be enough if we obtain $v_1=1$, viz. $G=K_1\times G_2$. Now suppose $v_1\geq 2$. Applying $n+1\geq 2v_1$ and $v_1\ge 2$ to (\ref{edgessum1}), we have
\begin{equation}\label{edgesumequality1}
e_{1}+e_{2}\leq (2-v_1)(2v_1-1)+v_1^{2}-v_1=-(v_1-2)^2+2.
\end{equation}
Notice the fact $e_1+e_2\ge 0$. It forces $v_1=2$ or $3$. Then our proof will be complete with the following cases.\vspace{2mm}

\noindent\emph{Case 1}. $v_1=2$.   Applying $v_1=2$ to (\ref{edgessum1}), we have $e_1+e_2=2$. Notice that $v_1=|VG_1|=2$ implies $e_1=|EG_1|\le 1$. Then $e_1=1$ or $0$.\vspace{2mm}

\noindent\emph{Case 1.1}. $e_{1}=1$.   Since $v_1=2$, it is clear that $G_1=K_2$. Then we have
\[G= K_{2}\times G_{2}=K_{1}\times (K_{1}\times G_{2}).\]
Since $ \mathrm{Spec}(U\times K_1)= \mathrm{Spec}(G)$ and $U$ is $L$-DS, by Lemma \ref{lemma1}, we obtain that $K_{1}\times G_{2}$ and $U$ are isomorphic. Clearly, $G= U\times K_1$.\vspace{2mm}

\noindent\emph{Case 1.2}. $e_{1}=0$.   It is clear that $G_1=2K_1$. Since $e_1+e_2=2$, then we have $e_2=|EG_2|=2$. Depending on two edges of $G_{2}$ either adjacent or not, $G_{2}$ may be isomorphic to $P_{3}+(n-4)K_{1}$ or $2P_{2}+(n-5)K_{1}$. Thus, we have
\[G= 2K_1\times \left(P_{3}+(n-4)K_{1}\right),  \text{~or~}  G= 2K_1\times \left(2P_{2}+(n-5)K_{1}\right).\]

\noindent\emph{Case 1.2.1}. $G= 2K_1\times \left(P_{3}+(n-4)K_{1}\right)$.   Since $ \mathrm{Spec}(P_3)=\{3,1,0\}$, applying Lemma \ref{product laplacian}, we have
\[ \mathrm{Spec}(G)=\{n+1,n-1,5,3,[2]^{n-4},0\}.\]
Since $ \mathrm{Spec}(G)= \mathrm{Spec}(U\times K_1)$, applying Lemma \ref{product laplacian} again, we obtain
\[ \mathrm{Spec}(U)=\{n-2,4,2,[1]^{n-4},0\}.\] By
Lemma \ref{spanning tree}, the number of the spanning trees of $U$
is $s(U)=\frac{8(n-2)}{n}$. It forces $n=4,8$ or $16$.\vspace{2mm}

\noindent\emph{Case 1.2.1.1}. $n=4$.   Clearly, we have $G= 2K_1\times P_{3}$. Notice that the unicyclic graph $U$ has $s(U)=\frac{8(n-2)}{n}=4$ spanning trees. It follows that the cycle of $U$ is $C_4$. But $|VU|=n=4$, then $U=C_4$.  On the other hand, it is easily seen that $2K_1\times P_3= K_1\times C_4$, viz. $G= K_1\times U$ in this case.\vspace{2mm}

\noindent\emph{Case 1.2.1.2}. $n=8$.   Since  $s(U)=\frac{8(n-2)}{n}=6$ and $U$ is unicyclic, then the cycle $C_{6}$ is a subgraph of $U$, i.e., $U\in \mathcal{U}(8,6)$. By Proposition \ref{8.1}, we have $d(v)+m(v)\leq 5.5$ for all $v\in VU$. Notice that the maximum Laplacian eigenvalue of $U$ is $\lambda(U)=n-2=6$. It is a contradiction by Lemma \ref{2nd max eigenvalue}.\vspace{2mm}

\noindent\emph{Case 1.2.1.3}. $n=16$.   Similar as \emph{Case 1.2.1.2}, we have $s(U)=7$ and $U\in\mathcal{U}(16,7)$. By Proposition \ref{8.1}, $d(v)+m(v)\leq12+\frac{2}{11}$ for all $v\in VU$. So it is a contradiction by Lemma \ref{2nd max eigenvalue} since the maximum Laplacian of $U$ is $\lambda(U)=n-2=14$ for the current case.\vspace{2mm}

\noindent\emph{Case 1.2.2}. $G= 2K_1\times \left(2P_{2}+(n-5)K_{1}\right)$.   Similar as \emph{Case 1.2.1}, applying Lemma \ref{product laplacian}, we obtain that
\[\mathrm{Spec}(G)=\{n+1,n-1,[4]^2,[2]^{n-4},0\}, \text{~and~}  \mathrm{Spec}(U)=\{n-2,[3]^2,[1]^{n-4},0\}.\]
It follows that the number of spanning trees of $U$ is $s(U)=\frac{9(n-2)}{n}$, so $n=6$, $9$, or $18$.\vspace{2mm}

\noindent\emph{Case 1.2.2.1}. $n=6$.   Similar as \emph{Case 1.2.1.2}, we have $U\in \mathcal{U}(6,6)$, which implies $U$ is
exactly the cycle $C_{6}$. By routine calculations, we can check that
\[\mathrm{Spec}(K_{1}\times  C_{6})=\mathrm{Spec}(2K_{1}\times (2P_{2}+ K_{1})).\]
But it is easily seen that $K_1\times C_6$ and $2K_{1}\times (2P_{2}+ K_{1})$ are not isomorphic, see Figure \ref{F1}. \vspace{2mm}

\noindent\emph{Case 1.2.2.2}. $n=9$.  Similar as \emph{Case 1.2.1.2}, we have $U\in\mathcal{U}(9,7)$, and then $d(v)+m(v)\leq 5.5$ for all $v\in VU$. But $\lambda(U)=7$ in this case, so it is a contradiction by Lemma \ref{2nd max eigenvalue}.\vspace{2mm}

\noindent\emph{Case 1.2.2.3}. $n=18$.   The arguments in this case is also similar as \emph{Case 1.2.1.2}. It is a contradiction by Lemma \ref{2nd max eigenvalue} since $U\in \mathcal{U}(18,8)$ and $m(v)+d(v)\leq 13+\frac{1}{6}$ for all $v\in VU$, but $\lambda(U)=16$.\vspace{2mm}

\noindent\emph{Case 2}. $v_1=3$.   Applying $v_1=3$ to (\ref{edgessum1}), we have $e_1+e_2=6-n$. Using the fact $e_1+e_2\ge 0$, then $n\leq 6$. Notice that $|VG|=n+1\geq 2|VG_1|=6$ implies $n\geq 5$. Thus, $n=5$ or $6$.
\vspace{2mm}

\noindent\emph{Case 2.1}. $n=6$.   Applying $v_1=3$ and $n=6$ to (\ref{edgessum1}), we can obtain $e_{1}+e_{2}=0$, viz. $e_1=e_2=0$.
Then we have $G=3K_{1}\times 4K_{1}$, whose Laplacian spectrum is $ \mathrm{Spec}(G)=\{7,[4]^{2},[3]^3,0\}$. Since $ \mathrm{Spec}(G)= \mathrm{Spec}(U\times K_1)$, applying Lemma \ref{product laplacian}, we have $ \mathrm{Spec}(U)=\{[3]^2,[2]^3,0\}$. Using Lemma \ref{spanning tree}, we have $s(U)=12$. However, $|VU|=n=6$, it follows that $s(U)\le 6$, a contradiction.\vspace{2mm}

\noindent\emph{Case 2.2}. $n=5$.   Our arguments are similar as \emph{Case 2.1}. When $n=5$, we have $e_1+e_2=1$, and then $G= (P_2+K_1)\times 3K_1$. Similar as \emph{Case 2.1}, we can easily obtain $s(U)=\frac{32}{5}$. Note the fact that $s(U)$ is an integer, a contradiction. This completes the proof.
\end{proof}

The following result is obvious from Lemmas \ref{unicyclic lemma} and \ref{product laplacian}. Indeed, $L$-cospectral graphs shown in Figure \ref{F1} have  also been given in \cite{kn:Zhang09}.

\begin{corollary}If $n\ne 6$, then
$C_n\times K_{m}$ is $L$-DS for all $m\geq 1$. Moreover, the $L$-cospectral graph of $C_{6}\times K_{m}$ is $H_1\times K_{m-1}$, where $H_1=2K_{1}\times (2P_{2}+ K_{1})$ as shown in Figure \ref{F1}.
\end{corollary}

\begin{theorem}\label{unicyclic}
If a unicyclic graph $U$ is $L$-DS and $U\ne C_6$, then the product $U\times K_m$ is $L$-DS for all positive integers $m$.
\end{theorem}
\begin{proof} The idea to prove this theorem is similar as the proof of Theorem \ref{treekm}. In the following, we borrow all of arguments and notations ahead of (\ref{numberedges}) in the proof of Theorem \ref{treekm}, except that the tree $T$ is replaced by the unicyclic graph $U$. In the following, we prove the theorem by induction on $m$. Note that the case $m=1$ is Lemma \ref{unicyclic lemma}. Now assume $m\ge 2$.  We are going to prove $v_{m}=1$ by contradiction. Suppose $v_{m}\geq 2$. Since $|EU|=|VU|=n$, instead of (\ref{numberedges}), we have
\begin{equation}\label{eq9}
    \sum_{i=0}^{m}e_{i}+\sum_{0\leq i<j\leq m}v_{i}v_{j}=n+mn+\frac{m(m-1)}{2}.
\end{equation}
Then by the same arguments as Theorem \ref{treekm}, instead of (\ref{sumes}), (\ref{syezss}), and (\ref{bianfanwei}), we have
\begin{eqnarray}
\sum_{i=0}^{m}e_{i}&=&\frac{1}{2}\left(\sum_{i=0}^{m}v_{i}^{2}-n^{2}-m\right)+n;\label{eq10}\\
\sum_{i=0}^{m}e_{i}&\le& (1-m)n+\frac{1}{2}(m^2+3m);\label{eq11}\\
 \sum_{i=0}^{m}e_{i}&\le& -\frac{1}{2}(m^2-m-4).\label{eq12}
\end{eqnarray}
Note that $\sum_{i=0}^{m}e_{i}\ge 0$. Applying the assumption $m\ge 2$ to (\ref{eq12}), we have $m=2$. Then (\ref{eq11}) becomes
\[e_{0}+e_1+e_2\le -n+5.\]
It follows that $n\le 5$. On the other hand, following the arguments of Theorem \ref{treekm}, we also have, similar as (\ref{eq5}), $n\ge m+2=4$. Combining them together, we have $n=4$ or $5$.\vspace{2mm}

\noindent\emph{Case 1}. $n=4$.   It is easily obtained that $v_{0}+v_{1}+v_{2}=m+n=6$. Recall the assumption $v_{0}\geq v_{1}\geq v_{2}\geq2$. It follows that $v_{0}=v_{1}=v_{2}=2$. Now applying $m=2$, $n=4$, and $v_{0}=v_{1}=v_{2}=2$ to (\ref{eq10}), we have $e_{0}+e_{1}+e_{2}=1$. It is easily seen that
\[G=K_{2}\times 2K_{1}\times 2K_{1}= C_{4}\times K_{2}.\]
Since $ \mathrm{Spec}(G)= \mathrm{Spec}(U\times K_2)$, by Lemma \ref{lemma1}, we have $U= C_4$, and then $G= U\times K_2$. \vspace{2mm}

\noindent\emph{Case 2.} $n=5$.   Clearly, $v_{0}+v_{1}+v_{2}=7$. Since $v_{0}\geq v_{1}\geq v_{2}\geq2$, then we have
\[v_{0}=3,~~v_{1}=2,~~v_{2}=2.\]
Applying these to (\ref{eq10}), we obtain $e_{0}=e_{1}=e_{2}=0$. It means that
\[G=3K_{1}\times  2K_{1}\times  2K_{1}.\]
From Lemma \ref{product laplacian}, by routine calculations, we have
\[\mathrm{Spec}(G)=\{[7]^2,[5]^2,[4]^2,0\}.\] Since $\mathrm{Spec}(G)= \mathrm{Spec}(U\times K_2)$, applying Lemma \ref{product laplacian} again, we have \[ \mathrm{Spec}(U)=\{[3]^2,[2]^2,0\}.\]
It is a contradiction since the number of spanning tree of $U$ is $s(U)=\frac{36}{5}$ by Lemma \ref{spanning tree}. So far, what we have obtained is $v_m=1$, i.e., $G_m=K_1$. Since $\mathrm{Spec}(G)=\mathrm{Spec}(U\times K_m)$, namely,
\[\mathrm{Spec}((G_0\times \cdots\times G_{m-1})\times K_1)= \mathrm{Spec}((U\times K_{m-1})\times K_1),\]
by Lemma \ref{product laplacian}, it is easy to obtain that
\[\mathrm{Spec}(G_0\times \cdots\times G_{m-1})= \mathrm{Spec}(U\times K_{m-1}).\]
From the induction hypothesis of $m-1$, we have
\[G_0\times \cdots\times G_{m-1}= U\times K_{m-1}.\]
Obviously, we have $G= U\times K_m$. This completes the proof.
\end{proof}

Up until now, there are only few unicyclic graphs have been proved to be $L$-DS graphs. For example, \emph{lollipop graph}, which is a graph obtained by attaching a pendant vertex of a path to a cycle, and graph $H(n;q,n_1,n_2)$ with order $n$, which contains a cycle $C_q$ and two hanging paths $P_{n_1}$ and $P_{n_2}$ attached at the same vertex of the cycle, are proved to be $L$-DS graph \cite{lollipop,kn:liu11}. Thus we can trivially get the following results.
\begin{corollary}
Let $G$ be the lollipop graph. Then $G\times K_{m}$ is $L$-DS for all positive integers $m$.
\end{corollary}
\begin{corollary}
Let $G=H(n;q,n_1,n_2)$. Then $G\times K_{m}$ is $L$-DS for all positive integers $m$.
\end{corollary}

\section{Laplacian spectral characterization of the products of bicyclic graphs and  complete graphs}\label{3333}
Recall that a bicyclic graph is a connected graph with two independent cycles. Strictly speaking, a connected graph $G=(VG,EG)$ is called bicyclic if $|EG|=|VG|+1$. From now on, we shall denote by $B$ the bicyclic graph to distinguish it from other graphs. This section is devoted to the study of Laplacian spectral characterization of the products of bicyclic graphs and complete graphs. The main result is that the products $B\times K_m$ are $L$-DS for all $L$-DS bicyclic graphs $B$ but one, $B=\Theta_{3,2,5}$, see Figure \ref{F2}. Before proceeding, we need some preparations.

Let $G_1=(VG_1, EG_1)$ and $G_2=(VG_2, EG_2)$ be two connected graphs. The  \emph{union} of $G_1$ and $G_2$ is defined to be
\[G_1\cup G_2=(VG_1\cup VG_2, EG_1\cup EG_2),\]
and the \emph{intersection} of $G_1$ and $G_2$ is defined to be
\[G_1\cap G_2=(VG_1\cap VG_2, EG_1\cap EG_2).\]
Let $C_r$ and $C_s$ be two cycles. If $C_r\cap C_s$ is a path $P_t$ of size $t\ge1$, then the graph union $C_r\cup C_s$ is denoted by $\Theta_{r,t,s}$. If $C_r\cap C_s$ is the empty graph, but $C_r$ and $C_s$ is connected by a path, then the union of $C_r$ and $C_s$, together with the path between them, is denoted by {\bf $\Theta_{r,0,s}$}. The subscription $t$ and $0$ represent the size of the intersection of two cycles. Graphs $\Theta_{r,t,s}$ for all $0\le t\le r\le s$ are called \emph{$\Theta$-graphs}. Informally speaking, the $\Theta$-graph is a bicyclic graph, either consisting of two cycles whose intersection is a path, or obtained by appending two disjoint cycles to two ends of a path. Note that the graph $\Theta_{r,t,s}$ with $2\le t\le r\le s$, are the same as graphs $\Theta_{r,r-t+2,r+s-2t+2}$ and $\Theta_{s,s-t+2,r+s-2t+2}$. For example, $\Theta_{3,2,4}=\Theta_{4,4,5}=\Theta_{3,3,5}$. To avoid this situation, we set $r,s\ge 2(t-1)$ for $t\ge 2$. Then we assume that
\begin{equation}\label{F9.2}
s\ge r \ge 2t-2\text{~~~~for~~} t\ge 0.
\end{equation}
 Clearly, a graph $B$ is bicyclic iff $B$ contains exactly one $\Theta$-graph. Denote by $\mathcal{B}(n,r,t,s)$ the collection of bicyclic graphs $B$ with $\Theta_{r,t,s}$ as a subgraph and $|VB|=n$. For $B\in \mathcal{B}(n,r,t,s)$, since \[|VB|\ge|V\Theta_{r,t,s}|\ge r+s-t\text{~~~~for~~} t\ge 0,\]
 then we have
\begin{equation}\label{F9.1}
n\ge r+s-t.
\end{equation}

Now, we give one proposition which will play an important rule in this section.

\begin{proposition}\label{9.2} Let $B\in \mathcal{B}(n,r,t,s)$ be a bicyclic graph, and denote $\alpha(B)$ the following number
\[\alpha(B)=\max\{d(v)+m(v)\mid v\in VB\}.\]
\noindent\emph{(1)} If $t=0$, then
\begin{equation}\label{F21}s(B)=rs,~~d(v)\leq n-r-s+3,~~\alpha(B)=n-r-s+4+\frac{4}{n-r-s+3}.\end{equation}

\noindent\emph{(2)} If $t=1$, then
\begin{equation}\label{F22}s(B)=rs,~~ d(v)\leq n-r-s+5,~~\alpha(B)=n-r-s+6+\frac{4}{n-r-s+5}.\end{equation}

\noindent\emph{(3)} If $t=2$, then
\begin{equation}\label{F23}s(B)=rs-1,~~d(v)\leq n-r-s+5,~~\alpha(B)= n-r-s+6+\frac{4}{n-r-s+5}.\end{equation}

\noindent\emph{(4)} If $t\ge 3$, then
\begin{equation}\label{F24}s(B)=rs-(t-1)^2,~~d(v)\leq n+t-r-s+3,~~\alpha(B)=n+t-r-s+4+\frac{3}{n+t-r-s+3}.\end{equation}
\end{proposition}

\begin{proof}In the following, we just give a detailed proof for the case $t=0$. The other cases will be followed by similar arguments. \vspace{2mm}

Assume $B\in \mathcal{B}(n,r,0,s)$. $B$ contains $\Theta_{r,0,s}$, and then two disjoint cycles $C_r$ and $C_s$. Each spanning tree of $B$ can be obtained by removing two edges from $B$, that is, one from $EC_r$ and the other from $EC_s$. It implies that the number of spanning trees of $B$ is $s(B)=rs$.\vspace{2mm}

To prove $d(v)\leq n-r-s+3$ for all $v\in VB$, we denote by $\Theta_{r,0,s}^\prime$ the $\Theta$-graph $\Theta_{r,0,s}$, whose two cycles $C_r$ and $C_s$ are connected by the path $P_2$. It is clear that any $\Theta_{r,0,s}$ contains at least $r+s$ vertices. Note that $|V\Theta_{r,0,s}|\ge r+s$ and the equality holds iff $\Theta_{r,0,s}=\Theta_{r,0,s}^\prime$. Then we have
\[|VB\setminus V\Theta_{r,0,s}|=|VB|-|V\Theta_{r,0,s}|\le n-r-s.\]
In order to make the maximum vertex degree of $B$ is $n-r-s+3$, $B$ must be obtained from $\Theta_{r,0,s}^\prime$ by attaching $n-r-s$ vertices to a vertex with degree $3$ in $\Theta_{r,0,s}^\prime$. Clearly, this is the maximum case.\vspace{2mm}

Now we are ready to prove the last equality of (\ref{F21}). If $t=0$, to make $m(v)+d(v)$ maximum, it is clearly required that $B$ contains $\Theta^\prime_{r,0,s}$ as a subgraph. Suppose $\alpha(B)=d(v_0)+m(v_0)$ for some $v_0\in VB$. This easily implies that $v_0$ must be an end of the path $P_2$ connecting two cycles of $\Theta^\prime_{r,0,s}$. Then we have $3\le d(v_0)\le n-r-s+3$. Denote by $n(v)=\sum_{\{u,v\}\in EG}d(u)$ the total degree of neighbors of $v$. Then we have
\[\alpha(B)=\max\left\{d(v_0)+\frac{n(v_0)}{d(v_0)}\mid 3\le d(v_0)\le n-r-s+3\right\}.\]
To make $d(v_0)+\frac{n(v_0)}{d(v_0)}$ as large as possible, all vertices in $VB\setminus V\Theta_{r,0,s}^\prime$ are, either incident with $v_0$, or incident with neighbors of $v_0$. It follows that $n(v_0)$ is a constant $n-r-s+7$. By some arithmetic calculations, we can obtain that, for $d(v_0)=n-r-s+3$,
\[\alpha(B)=n-r-s+3+\frac{n-r-s+7}{n-r-s+3}=n-r-s+4+\frac{4}{n-r-s+3}.\]
This completes the proof of case $t=0$.
\end{proof}

\begin{lemma}\label{bicyclic lemma}
If a bicyclic graph $B$ is $L$-DS and $B\neq \Theta_{3,2,5}$, then $B\times K_{1}$ is $L$-DS. Furthermore, $\Theta_{3,2,5}\times K_{1}$ is $L$-cospectral to $2K_{1}\times  (P_{4}+K_{1})$, see Figure \rm{\ref{F2}}.
\end{lemma}

\begin{proof}
We shall use similar arguments and notations as Lemma \ref{lemma2} or \ref{unicyclic lemma}. Then we can assume that $G=G_1\times G_2$ is a graph $L$-cospectral to $B\times K_{1}$ and
\[|VB|=n,~~|VG_1|=v_1,~~|VG_2|=v_2,~~|EG_{1}|=e_1,~~|EG_{2}|=e_2.\]
We further assume $n+1\geq 2v_1$. By counting the edges of both $G_1\times G_2$ and $B\times K_{1}$, Lemma \ref{Laplacian spectrum} implies
\begin{equation}\label{kkbian}
    e_{1}+e_{2}=(2-v_1)n+v_1^2-v_1+1.
\end{equation}
By Lemma \ref{lemma1}, we are required to prove $v_1=1$. Suppose $v_1\geq 2$, applying $n+1\geq 2v_1$ to (\ref{kkbian}), we have
\begin{equation}\label{jjhebian}
e_{1}+e_{2}\leq (2-v_1)(2v_1-1)+v_1^{2}-v_1+1= -v_1^2+4v_1-1.
\end{equation}
The fact $e_1+e_2\ge 0$ implies that $v_1=2$ or $3$.\vspace{2mm}

\noindent\emph{Case 1.} $v_1=2$.  Clearly, $v_1=2$ implies that $e_1\le 1$. If $e_1=1$, the graph $G_1=K_1\times K_1$, and then $G$ can be written as $K_1\times (K_1\times G_2)$, which is clear from Lemma \ref{lemma1}. If $e_1=0$, then $G_1=2K_1$. Substituting $v_1=2$ into (\ref{kkbian}), we have $e_1+e_2=3$. Then we obtain $e_2=3$. According to $v_2=n-1$ and $e_2=3$, $G_{2}$ has to be one of the following graphs
\[3P_{2}+(n-7)K_1,~~ P_{2}+P_{3}+(n-6)K_1,\]
\[P_{4}+(n-5)K_1,~~K_1\times 3K_1+(n-5)K_1,~~C_{3}+(n-4)K_1.\]
Then consider the following cases. \vspace{2mm}

\noindent\emph{Case 1.1.} $G_{2}=3P_{2}+(n-7)K_1$.   Since $G_1=2K_1$, applying Lemma \ref{product laplacian} and by routine calculations, we can obtain
\[ \mathrm{Spec}(G)=\{n+1,n-1,[4]^3,[2]^{n-5},0\}.\] It follows that
\begin{equation}\label{F4000} \mathrm{Spec}(B)=\{n-2,[3]^3,[1]^{n-5},0\}.\end{equation}
The number of spanning trees of $B$ is given by
\[s(B)=27-\frac{54}{n}.\]
Thus $n\mid 54$. On the other hand, it is clear that $n\ge 7$ since $G_{2}=3P_{2}+(n-7)K_1$. Hence, $n=9,18,27$, or $54$. In the following, we shall rule out them case by case.\vspace{2mm}

\noindent\emph{Case 1.1.1.} $n=9$.   From (\ref{F4000}), the maximum Laplacian eigenvalue of $B$ is $7$. By Lemmas \ref{spanning tree} and \ref{2nd max eigenvalue}, it is easily obtained that
\[\alpha(B)\ge 7,~~s(B)=21.\]
Assume $B\in \mathcal{B}(9,r,t,s)$, i.e., $B$ contains $\Theta_{r,t,s}$ as a subgraph. Then $r+s\le 9+t$ by (\ref{F9.1}). If $t=0$, (\ref{F21}) implies $s(B)=rs=21$ which is impossible for $r+s\le 9$. If $t=1$, $(\ref{F22})$ implies $rs=21$. It follows that $r=3,s=7$. Then using $(\ref{F22})$ again, $\alpha(B)\le 6$, a contradiction. If $t=2$, $(\ref{F23})$ implies $s(B)=rs-1=21$ which has no solution since $s\geq r\geq 3$. If $t\ge 3$, from $(\ref{F24})$, we have
$\alpha(B)\leq n+t-r-s+5$. Applying $r+s\ge 4t-4$ in (\ref{F9.2}), we obtain $\alpha(B)\le n+9-3t$. Note that $\alpha(B)\ge n-2$. Then we have the following inequality
\begin{equation}\label{F9.9}
n-2\le \alpha(B) \le n+9-3t.
\end{equation}
Combined with $t\ge 3$, it forces $t=3$. From $s(B)=21=rs-(t-1)^2$, we obtain $r=s=5$. Applying $n=9,t=3,r=s=5$ to (\ref{F24}), it is easily obtained that $\alpha(B)\le \frac{33}{5}<7$, which is a contradiction. Hence, we proved $n\neq 9$\vspace{2mm}

\noindent\emph{Case 1.1.2.} $n=18,27$, or $54$.   We just give the case $n=18$ in details, but skip the cases $n=27,54$, since all arguments are the same. Suppose that $n=18$ and $B\in \mathcal{B}(18,r,t,s)$. Then the maximum Laplacian eigenvalue of $B$ is $n-2=16$. Lemmas \ref{spanning tree} and \ref{2nd max eigenvalue} imply
\[\alpha(B)\ge 16,~~s(B)=24.\]
If $t=0$ or $1$,  then $s(B)=rs=24$. It implies that $r=4,s=6$ or $r=3,s=8$. Applying (\ref{F21}) and (\ref{F22}) to all cases of $t,r,s$, we obtain
$\alpha(B)<15$, contradictions. If $t=2$, then $s(B)=rs-1=24$, i.e., $r=s=5$. By (\ref{F23}), we get $\alpha(B)<15$, a contradiction. If $t\ge 3$, then (\ref{F9.9}) implies $t=3$. Since $s(B)=rs-(t-1)^2=24$ and $s\ge r\ge 2t-2$, we obtain $t=3,r=4,s=7$. By (\ref{F24}), we have $\alpha(B)<15$, a contradiction.\vspace{2mm}

\noindent\emph{Case 1.2.} $G_{2}=P_{2}+P_{3}+(n-6)K_1$.   Since $\mathrm{Spec}(P_{3})=\{3,1,0\}$, by Lemma \ref{product laplacian}, we have
\[ \mathrm{Spec}(G)=\{n+1,n-1,5,4,3,[2]^{n-5},0\},\]
and then
\[\mathrm{Spec}(B)=\{n-2,4,3,2,[1]^{n-5},0\}.\]
From Lemma \ref{spanning tree}, $s(B)=24-\frac{48}{n}$. Note that $n\ge 6$ for $G_{2}=P_{2}+P_{3}+(n-6)K_1$. It follows that $n=6,8,12,16,24,$ or $48$.\vspace{2mm}

\noindent\emph{Case 1.2.1.} $n=6$.   Then the maximum Laplacian eigenvalue of $B$ is $n-2=4$. Lemma \ref{min} implies that the maximum vertex degree of $B$, denoted $\Delta(B)$, is at most $3$ and not identical $3$. Namely, all vertices degree of $B$ is at most $2$. It is a contradiction with that $B$ is a bicyclic graph.
\vspace{2mm}

\noindent\emph{Case 1.2.2.} $n=8$.   The maximum Laplacian eigenvalue of $B$ is $n-2=6$. Lemmas \ref{spanning tree} and \ref{2nd max eigenvalue} imply that
\[\alpha(B)\ge 6,~~s(B)=18.\]
If $t=0$ or $1$, then $s(B)=rs=18$. It follows that $r=3,s=6$. Then (\ref{F21}) and (\ref{F22}) imply $\alpha(B)<6$, contradictions. If $t=2$, then $s(B)=rs-1=18$ which has no solutions of $r,s$. If $t\geq 3$, then (\ref{F9.9}) implies $t=3$. It is impossible since $s(B)=rs-(t-1)^2=18$ and $s\ge r\ge 2t-2=4$.
\vspace{2mm}

\noindent\emph{Case 1.2.3.} $n=12,16,24,$ or $48$.   We only show $n\neq 12$ in details. The others are similar. Suppose $n=12$. Then Lemmas \ref{spanning tree} and \ref{2nd max eigenvalue} imply that
\[\alpha(B)\ge 10,~~s(B)=20.\]
If $t=0$ or $1$, then $s(B)=rs=20$. It follows that $r=4,s=5$. (\ref{F21}) and (\ref{F22}) imply $\alpha(B)<10$, contradictions. If $t=2$, then $s(B)=rs-1=20$, i.e., $r=3,s=7$. $(\ref{F23})$ implies $\alpha(B)<9$, a contradiction. If $t\geq 3$, then (\ref{F9.9}) implies $t=3$. Since $s(B)=rs-(t-1)^2=20$, we obtain $t=3,r=4,s=6$. By (\ref{F24}), we have
$\alpha(B)<10$, a contradiction.
\vspace{2mm}

\noindent\emph{Case 1.3.} $G_{2}=P_{4}+(n-5)K_1$.   Then $n\geq 5$. By routine calculations as above, we can obtain
\[\mathrm{Spec}(B)=\{n-2,3+\sqrt{2},3,3-\sqrt{2},[1]^{n-5},0\},\]
and then $s(B)=21-\dfrac{42}{n}$. It follows that $n$ must be $6,7,14,21,$ or $42$.
\vspace{2mm}

\noindent\emph{Case 1.3.1.} $n=6$.   It follows that $s(B)=14$. By similar arguments as above, we will obtain $t=2,r=3,s=5$ which can not be ruled out. Then $B$ consists of two cycles, $C_{3}$ and $C_{5}$, whose intersection is the path $P_2$, i.e., $B=\Theta_{3,2,5}$. Indeed, by routine calculations, we obtain that the Laplacian spectrum of $\Theta_{3,2,5}$ is exactly $\{3+\sqrt{2},4,3-\sqrt{2},3,1,0\}$. Hence, $\Theta_{3,2,5}\times  K_{1}$ is $L$-cospectral to $2K_{1}\times (P_{4}+K_{1})$, but not isomorphic, see Figure \ref{F2}.
\vspace{2mm}

\noindent\emph{Case 1.3.2.} $n=7$.    Note that the maximum Laplacian eigenvalue of $B$ is $5$. Applying Lemmas  \ref{spanning tree}, \ref{2nd max eigenvalue} and  \ref{min}, we obtain \[\alpha(B)\geq 5,~~\Delta(B)\le 3,~~s(B)=15, \]
 where $\Delta(B)$ is the maximum vertex degree of $B$. Thus $t\ne 1$, otherwise, $B$ has a vertex of degree at least $4$. If $t=0$, then (\ref{F21}) implies $s(B)=rs=15$, a contradiction to (\ref{F9.1}). If $t=2$, then (\ref{F23}) implies $s(B)=rs-1=15$, i.e., $r=s=4$. Since the maximum vertex degree of $B$ is $3$, combining with $n=7,t=2,r=s=4$, we obtain that the degree sequence of $B$ is $(3,3,3,2,2,2,1)$, denoted by \[\deg(B)=([3]^3,[2]^3,1),\]
where $[a]^b$ is a sequence of constant $a$ with multiplicity $b$.
It follows that
\[\deg(B\times  K_1)=(7,[4]^3,[3]^3,2).\]
On the other hand, since $G_1=2K_1$ and $G_{2}=P_{4}+(n-5)K_1=P_{4}+2K_1$, we have
\[\deg(G_1\times G_2)=([6]^2,[4]^2,[3]^2,[2]^2).\]
But it is easily checked that $7^2+3\cdot4^2+3\cdot3^2+2^2\neq 2\cdot(6^2+4^2+3^2+2^2)$, a contradiction to Lemma \ref{Laplacian spectrum}. If $t\geq 3$, then (\ref{F9.9}) implies $t=3$. It is impossible since $rs=s(B)+(t-1)^2=19$.\vspace{2mm}

\noindent\emph{Case 1.3.3.} $n=14,21$, or $42$.   We only disprove $n=42$ in details. The others are similar. Suppose $n=42$. Lemmas  \ref{spanning tree} and \ref{2nd max eigenvalue} imply that
\[\alpha(B)\ge 40,~~s(B)=20.\]
If $t=0$ or $1$,  then $s(B)=rs=20$, and then $r=4,s=5$. Applying (\ref{F21}) and (\ref{F22}), we obtain $\alpha(B)<40$,
contradictions. If $t=2$, then $s(B)=rs-1=20$, i.e., $r=3,s=7$. (\ref{F23}) implies $\alpha(B)<39$, a contradiction. If $t\geq 3$, then (\ref{F9.9}) implies $t=3$. Since $s(B)=rs-(t-1)^2=20$, we obtain $t=3,r=4,s=6$. By (\ref{F24}), we have $\alpha(B)<40$, a contradiction.
\vspace{2mm}

\noindent\emph{Case 1.4.} $G_2=K_1\times 3K_1+(n-5)K_1$.   Similar as above, we have $n\geq 5$ and
\[\mathrm{Spec(B)}=\{n-2,5,[2]^2,[1]^{n-5},0\},\]
and then the number of spanning trees of $B$ is $s(B)=20-\frac{40}{n}$. It forces $n=5,8,10,20,$ or $40$.
\vspace{2mm}

\noindent\emph{Case 1.4.1.} $n=5$.   Since $ \mathrm{Spec}(B)=\{5,3,[2]^2,0\}$, Lemma \ref{product graphs} implies that $B$ is the product of two graphs, say $B=B_1\times B_2$. If $B_1=K_1$, then we have $ \mathrm{Spec}(B_2)=\{2,[1]^2,0\}$. It follows that the number of spanning tree is $s(B_2)=\frac{2}{4}$, a contradiction. If $|VB_1|$=2, then $|VB_2|=3$. Notice that the second largest Laplacian eigenvalue of $B$ is 3, by Lemma \ref{product laplacian}, the maximum  Laplacian eigenvalue of $B_1$ is 0. Thus, $B_1=2K_1$, and then $B_2=3K_1$. It is obvious that $B\times K_1=G_1\times G_2$.\vspace{2mm}

\noindent\emph{Case 1.4.2.} $n=8$.   Note that the maximum Laplacian eigenvalue of $B$ is $6$. Applying Lemmas  \ref{spanning tree}, \ref{2nd max eigenvalue} and  \ref{min}, we obtain \[\alpha(B)\geq 6,~~\Delta(B)\le 4,~~s(B)=15.\]
If $t=0$, then $s(B)=rs=15$, i.e., $r=3,s=5$. Applying $(\ref{F21})$, we obtain $\alpha(B)<6$, a contradiction. If $t=1$, then $s(B)=rs=15$, i.e., $r=3,s=5$. It is easily obtained that the degree sequence of $B$ is \[\deg(B)=(4,3,[2]^5,1),\]
then the degree sequence of $B\times K_1$ is
\[\deg(B\times K_1)=(8,5,4,[3]^5,2),\]
whose square sum is $154$. But the degree sequence of $G_1\times G_2$ is
\[\deg(G_1\times G_2)=([7]^2,5,[3]^3,[2]^3),\]
whose square sum is $162\neq154$, a contradiction to Lemma \ref{Laplacian spectrum}, since $B\times K_1$ and $G_1\times G_2$ are $L$-cospectral. If $t=2$,  then $s(B)=rs-1=15$ implies $r=4,s=4$, i.e., $\Theta_{4,2,4}$
is a subgraph of $B$. Note that $|VB|-|V\Theta_{4,2,4}|=2$. Following the idea of the proof for Proposition \ref{fact1}, the square sum of the degree sequence of $B\times K_1$ is maximum iff  $\deg(B)=([4]^2,[2]^4,[1]^2)$. It follows that
\[\deg(B\times K_1)=(8,[5]^2,[3]^4,[2]^2),\]
whose square sum is $158$. But the square sum of the degree sequence of $G_1\times G_2$ is $162$, a contradiction to Lemma \ref{Laplacian spectrum}. If $t\geq 3$, then (\ref{F9.9}) implies $t=3$.  It is impossible since $rs=s(B)+(t-1)^2=19$.\vspace{2mm}

\noindent\emph{Case 1.4.3.} $n=10$. Note that the maximum Laplacian eigenvalue of $B$ is $8$. Applying Lemmas  \ref{spanning tree}, \ref{2nd max eigenvalue} and  \ref{min},  we obtain \[\alpha(B)\geq 8, ~~\Delta(B)\le 6,~~s(B)=16, \]
If $t=0$, then $s(B)=rs=16$, i.e., $r=4,s=4$. From (\ref{F21}), we obtain $\alpha(B)<8$, a contradiction. If $t=1$, then $s(B)=rs=16$, and then $r=s=4$. Thus $\Theta_{4,1,4}$ is a subgraph of $B$. Note that $|VB|-|V\Theta_{4,1,4}|=3$. Then the square sum of the degree sequence of $B\times K_1$ is maximum iff the degree sequence of $B$ is $(6,3,[2]^5,[1]^3)$. It follows that
\[\deg(B\times K_1)=(10,7,4,[3]^5,[2]^3),\]
whose square sum is $222$. Namely, the maximum square sum of the degree sequence of $B\times K_1$ is 222. But the square sum of the degree sequence of $G_1\times G_2$ is 234, a contradiction to Lemma \ref{Laplacian spectrum}. If $t=2$, then $s(B)=rs-1=16$ has no solution. If $t\geq 3$, then (\ref{F9.9}) implies $t=3$. Since $s(B)=rs-(t-1)^2=16$, we have $r=4,s=5$, i.e., $\Theta_{4,3,5}$
is a subgraph of $B$. Also consider the square sum of the degree sequence of $B\times K_1$. We can obtain that its maximum value is 226 only when \[\deg(B\times K_1)=(10,7,5,[3]^4,[2]^4).\]
But the square sum of the degree sequence of $G_1\times G_2$ is 234, a contradiction to Lemma \ref{Laplacian spectrum}.
\vspace{2mm}

\noindent\emph{Case 1.4.4.} $n=20$ or $40$.   We only disprove $n=20$ in details. The argument to disprove $n=40$ will be similar. Suppose $n=20$. Note that the maximum Laplacian eigenvalue of $B$ is $18$. Applying Lemmas  \ref{spanning tree}, \ref{2nd max eigenvalue} and  \ref{min}, we obtain \[\alpha(B)\geq 18,~~\Delta(B)\le 16, ~~ s(B)=18.\]
If $t=0$ or $1$, then $s(B)=rs=18$, i.e., $r=3,s=6$. Applying (\ref{F21}) and (\ref{F22}), we obtain $\alpha(B)<18$, a contradiction. If $t=2$, then $s(B)=rs-1=18$ has no solution. If $t\geq 3$, then we can also obtain (\ref{F9.9}), and then $t=3$. Since $s(B)=rs-(t-1)^2=18$, there is no solution.\vspace{2mm}

\noindent\emph{Case 1.5.} $G_{2}=C_{3}+(n-4)K_1$.   Clearly, $n\ge 4$. Since the Laplacian spectrum of $C_{3}$ is $\{[3]^2,0\}$, we can obtain
\[\mathrm{Spec}(B)=\{n-2,[4]^2,[1]^{n-4},0\}.\]
It follows that the number of spanning trees of $B$ is $s(B)=16-\frac{32}{n}$. Hence, $n=4,8,16$, or $32$.\vspace{2mm}\\
\emph{Case 1.5.1.} $n=4$.   It follows that $ \mathrm{Spec}(B)=\{[4]^2,2,0\}$. Then we have $B=K_1\times K_1\times 2K_1$ by Lemma \ref{product graphs}. It is easily seen that $B\times K_1= G_1\times G_2$.\vspace{2mm}

\noindent\emph{Case 1.5.2.} $n=8$.   Note that the maximum Laplacian eigenvalue of $B$ is $6$.  Applying Lemmas  \ref{spanning tree}, \ref{2nd max eigenvalue} and  \ref{min}, we obtain \[\alpha(B)\geq 6,~~\Delta(B)\le 4,~~s(B)=12. \]
If $t=0$, then $s(B)=rs=12$ implies that $r=3,s=4$, and then $\alpha(B)\leq 6$ by (\ref{F21}). Thus,
$\alpha(B)=6$, which forces the degree sequence of $B$ is $\deg(B)=(4,3,[2]^5,1)$. Then the degree sequence of $B\times K_1$ is $\deg(B\times  K_1)=(8,5,4,[3]^5,2)$, whose square sum is $154$. On the other hand, since $G_2=C_3+4K_1$ and $G_1=2K_1$, then the degree sequence of $G_1\times G_2$ is $\deg(G_1\times G_2)=([7]^2,[4]^3,[2]^4)$, whose square sum is $162$. It is a contradiction by Lemma \ref{Laplacian spectrum}. If $t=1$, then we also have $r=3,s=4$, i.e., $\Theta_{3,1,4}$ is a subgraph of $B$. Since $\Delta(B)\le 4$, then the square sum of the degree sequence of $B\times K_1$ is maximum iff $\deg(B)=([4]^2,[2]^4,[1]^2)$. It
follows that $\deg(B\times K_1)=(8,[5]^2,[3]^4,[2]^2)$ whose square sum is $158<162$, a contradiction to Lemma \ref{Laplacian spectrum}. If $t=2$, then $s(B)=rs-1=12$ has no solution. If $t\geq 3$, then we can also obtain (\ref{F9.9}), and then $t=3$. Since $s(B)=rs-(t-1)^2=12$, then $t=3, r=s=4$. Then $B$ contains $\Theta_{3,2,4}$ as a subgraph. Since $\Delta(B)\le 4$, then the square sum of the degree sequence of $B\times K_1$ is maximum iff $\deg(B)=([4]^2,3,[2]^2,[1]^3)$. It follows that $\deg(B\times K_1)=(8,[5]^2,4,[3]^2,[2]^3)$, whose square sum is $160<162$, a contradiction.
\vspace{2mm}

\noindent\emph{Case 1.5.3.} $n=16$.   Note that the maximum Laplacian eigenvalue of $B$ is $14$. Applying Lemmas  \ref{spanning tree}, \ref{2nd max eigenvalue} and \ref{min}, we obtain \[\alpha(B)\geq 14,~~\Delta(B)\le 12,~~s(B)=14. \]
If $t=0$ or $1$, then $s(B)=rs=14$ has no solution. If $t=2$, then $s(B)=rs-1=14$ implies $r=3,s=5$, i.e., $\Theta_{3,2,5}$ is a subgraph of $B$. Since $|VB|-|V\Theta_{3,2,5}|=0$ and and $\Delta(B)\le 12$, then the square sum of $\deg(B\times K_1)$ is maximum iff $\deg(B)=(12,4,[2]^4,[1]^{10})$. It follows that the square sum of $\deg(B\times K_1)$ is $526$. But the square sum of $\deg(G_1\times G_2)$ is $546$. It is a contradiction to Lemma \ref{Laplacian spectrum}. If $t\geq 3$, then (\ref{F9.9}) implies $t=3$. Since $s(B)=rs-(t-1)^2=14$,  we have $r=3,s=6$, a contradiction to (\ref{F9.2}).
\vspace{2mm}

\noindent\emph{Case 1.5.4.} $n=32$.   Note that the maximum Laplacian eigenvalue of $B$ is $30$. Applying Lemmas  \ref{spanning tree}, \ref{2nd max eigenvalue} and \ref{min}, we obtain
\[\alpha(B)\geq 30,~~\Delta(B)\le 28,~~s(B)=15. \]
If $t=0$, then $s(B)=rs=15$, i.e., $r=3,s=5$. Applying (\ref{F21}), we obtain $\alpha(B)<29$, a contradiction. If $t=1$, then $s(B)=rs=15$ implies $r=3,s=5$, i.e., $\Theta_{3,1,5}$ is a subgraph of $B$. Then, the square sum of the degree sequence of $B\times K_1$ is maximum iff $\deg(B)=(28,3,[2]^5,[1]^{25})$. It follows that the degree sequence of $B\times K_1$ is $(32,29,5,[3]^4,[2]^{26})$, whose square sum is $2014$. But the square sum of the degree sequence of $G_1\times G_2$ is $2082$, a contradiction to Lemma \ref{Laplacian spectrum}. If $t\geq 3$, then (\ref{F9.9}) implies $t=3$. Then $s(B)=rs-(t-1)^2=15$ has no solution.
\vspace{2mm}

\noindent\emph{Case 2.} $v_1=3$.   Substituting $v_1=3$ into (\ref{kkbian}), we have $0\leq e_{1}+e_{2}=-n+7$, i.e., $n\le 7$. From the assumption that $n\ge 2v_1-1$, we have $n\ge 5$. Then, $n=5,6$, or $7$.
\vspace{2mm}

\noindent\emph{Case 2.1.} $n=5$.   Substituting $v_1=3$ and $n=5$ into (\ref{kkbian}), we have $e_{1}+e_{2}=2$. Clearly, $v_{2}=n+1-v_1=3$. It follows that
\[G= P_{3}\times 3K_{1},  \text{~~or~~}   G=(P_{2}+K_{1})\times (P_{2}+K_{1}).\]
Since $G= P_{3}\times 3K_{1}=K_{1}\times (2K_{1}\times 3K_{1})$, Lemma \ref{lemma1} implies $B=2K_{1}\times 3K_{1}$. Now consider $G=(P_{2}+K_{1})\times (P_{2}+K_{1})$. By Lemma \ref{product laplacian}, the Laplacian spectrum of $G$ is $\{6,[5]^2,[3]^2,0\}$. Then the Laplacian spectrum of $B$ is $\{[4]^2,[2]^2,0\}$, and then $s(B)=64/5$, a contradiction.
\vspace{2mm}\\
\emph{Case 2.2.} $n=6$.   Substituting $v_1=3$ and $n=6$ into
(\ref{kkbian}), we have $e_{1}+e_{2}=1$. Clearly, $v_{2}=4$. It follows that
\[G=3K_{1}\times (P_{2}+2K_{1}),  \text{~~or~~}   G=4K_{1}\times (P_{2}+K_{1}).\]
\emph{Case 2.2.1.} $G=3K_{1}\times (P_{2}+2K_{1})$.  By Lemma \ref{product laplacian}, the Laplacian spectrum of $B$ is $\{4,[3]^2,[2]^2,0\}$. By Lemma \ref{min}, we have $\Delta(B)\leq 2$, a contradiction to that $B$ is a bicyclic graph.
\vspace{2mm}\\
\emph{Case 2.2.2.} $G=4K_{1}\times (P_{2}+K_{1})$.   By Lemma \ref{product laplacian}, the Laplacian spectrum of $B$ is
$\{5,3,[2]^3,0\}$. Then Lemma \ref{spanning tree} implies that $s(B)=20$. Suppose $B\in \mathcal{B}(6,r,t,s)$. If $t=0$, then $s(B)=rs=20$, and then $r=4,s=5$, a contradiction to (\ref{F9.1}). If $t\geq 1$, then $\alpha(B)=rs-(t-1)^2=20$ and $r+s\leq n+t=t+6$ has no solution, a contradiction.
\vspace{2mm}\\
\emph{Case 2.3.} $n=7$.   Substituting $v_1=3$ and $n=7$ into (\ref{kkbian}), we have $e_{1}+e_{2}=0$. Clearly, $v_{2}=5$. Then $G=3K_{1}\times5K_{1}$. By Lemma \ref{product laplacian}, the Laplacian spectrum of $B$ is $\{[4]^2,[2]^4,0\}$. Then $s(B)=196/7$, a contradiction.\vspace{2mm}\\
So far, we can conclude that, for all bicyclic graphs but $\Theta_{3,2,5}$ as in \emph{Case 1.3.1}, we have $v_1=1$. The proof is complete by Lemma \ref{lemma1}.
\end{proof}

From \emph{Case 1.3.1} of Lemma \ref{bicyclic lemma} and Lemma \ref{product laplacian}, it is trivial to get the following result.
\begin{corollary}
Graphs $\Theta_{3,2,5}\times K_m$ and $H_2\times K_{m-1}$ are $L$-cospectral for all positive integers $m$, where $H_2$ is given in Figure \rm{\ref{F2}}.
\end{corollary}

In the following, we will use induction to prove the last main result.
\begin{theorem}\label{bicyclic}
If a bicyclic graph $B$ is $L$-DS and $B\ne \Theta_{3,2,5}$, then the product $B\times K_m$ is $L$-DS for all positive integer $m$.
\end{theorem}
\begin{proof}The idea to prove this theorem is similar as the proof of Theorem \ref{treekm} or \ref{unicyclic}. We repeat some arguments of Theorem \ref{treekm}. The statement for $m=1$ is given in Lemma \ref{bicyclic lemma}. Now we assume $m\ge 2$. Let $G$ be a graph $L$-cospectral to $B\times K_{m}=B\times \underbrace{K_1\times K_1\times\cdots\times K_1}_m$. By Lemma \ref{product graphs}, $G$ can be written as
\[G=G_{0}\times G_{1}\times\cdots\times G_{m}.\]
Fix notations as follows,
\begin{equation}\label{F40}n=|VB|,~~e_{i}=|EG_{i}|,~~v_{i}=|VG_{i}|  \text{~~~~for~~} i=0,1,\dots,m.\end{equation}
Without loss of generality, assume $v_{0}\geq v_{1}\geq\cdots\geq v_{m}$. It is obvious  that $\sum_{i=0}^{m}v_i=n+m$ by Lemma \ref{Laplacian spectrum}. In the following, we are going to prove $v_{m}=1$ by contradiction. Now suppose $v_{m}\geq 2$. It follows that $v_i\ge 2$ for all $i=0,\ldots,m$. Then we have $m+n=\sum_{i=0}^{m}v_i\ge 2(m+1)$, so $n\ge m+2$. For convenience, we list those conclusions we have obtained,
\begin{equation}\label{F41}
m\ge 2,~~v_0\ge\cdots\ge v_m\ge 2,~~m+n=\sum_{i=0}^{m}v_i,~~n\ge m+2.
\end{equation}
 Combining $v_0\ge\cdots\ge v_m\ge 2$ with $\sum_{i=0}^{m}v_i=n+m$, by Proposition \ref{fact1}, we have
\begin{equation}\label{F42}
\sum_{i=0}^{m}v_i^2\le (n-m)^2+4m.
\end{equation}
Since $ \mathrm{Spec}(G)= \mathrm{Spec}(T\times K_m)$, Lemma \ref{Laplacian spectrum} implies that $G$ and $T\times K_m$ have the same number of edges. Counting the edges of both $G$ and $T\times K_{m}$, we have
\begin{equation}\label{F43}
\sum_{i=0}^{m}e_{i}+\sum_{0\leq i<j\leq m}v_{i}v_{j}=n+1+mn+\frac{m(m-1)}{2}.
\end{equation}
Applying $\sum_{i=0}^{m} v_i=n+m$ to (\ref{F43}), we have
\begin{eqnarray}\label{F44}
\sum_{i=0}^{m}e_{i}=\frac{1}{2}\left(\sum_{i=0}^{m}v_{i}^{2}-n^{2}-m\right)+n+1.
\end{eqnarray}
Applying (\ref{F42}) to (\ref{F44}), it follows that
\begin{eqnarray}\label{F45}
\sum_{i=0}^{m}e_{i}\leq (1-m)n+\frac{1}{2}(m^2+3m)+1.
\end{eqnarray}
Note that, from (\ref{F41}), we have $1-m<0$ and $n\ge m+2$. Then (\ref{F45}) implies
\begin{eqnarray}\label{F46}
\sum_{i=0}^{m}e_{i}\leq -\frac{1}{2}(m^2-m-6).
\end{eqnarray}
Applying the fact $\sum_{i=0}^{m}e_i\ge 0$ to (\ref{F46}), it is easily obtained that $m=2$ or $3$.
\vspace{2mm}

\noindent\emph{Case 1.} $m=3$.   Substituting $m=3$ into (\ref{F45}) and (\ref{F46}), we have
\[0=\sum_{i=0}^{m}e_{i}\leq -2n+10.\]
It follows that $n\leq 5$. But $n\geq m+2=5$ by (\ref{F41}). Hence, $n=5$. It follows that $|VG|=\sum_{i=0}^{3}v_i=8$. Since $v_0\ge\cdots\ge v_3\ge 2$, then we have
\[v_{0}=v_{1}=v_{2}=v_{3}=2,  \text{~and~}   e_{0}=e_{1}=e_{2}=e_{3}=0.\]
Thus, $G=2K_{1}\times 2K_{1}\times 2K_{1}\times 2K_{1}$. By Lemma \ref{product laplacian}, the minimal nonzero Laplacian eigenvalue of $G$ is 2. But the minimal nonzero Laplacian eigenvalue of $B\times K_{3}$ is at least 3, a contradiction.
\vspace{2mm}

\noindent\emph{Case 2.} $m=2$.   By (\ref{F41}), we have  $n\geq m+2=4$. On the other hand, (\ref{F46}) implies that \[0\le\sum_{i=0}^{m}e_{i}\le -n+6.\] Then we have $n=4,5$, or $6$.
\vspace{2mm}\\
\emph{Case 2.1.} $n=4$.   Then $|VG|=v_{0}+v_{1}+v_{2}=6$ and $v_{0}\geq v_{1}\geq v_{2}\geq2$. It follows that $v_{0}=v_{1}=v_{2}=2$. By (\ref{F43}), we have $e_{0}+e_{1}+e_{2}=1$. It follows that
\[G= K_{2}\times K_{2}\times 2K_{1}=K_{1}\times K_{3}\times 2K_{1}.\]
By Lemma \ref{lemma1}, we have $G=B\times K_3$.
\vspace{2mm}

\noindent\emph{Case 2.2.} $n=5$.   Similar as above, we can obtain that
\[v_{0}=3, ~~v_{1}=v_{2}=2, \text{~and~} e_{0}+e_{1}+e_{2}=1.\]
Namely,
\begin{eqnarray*}
 G=3K_{1}\times 2K_{1}\times K_{2}, \text{~or~} G=(P_{2}+K_{1})\times 2K_{1}\times 2K_{1}.
\end{eqnarray*}
If $G=3K_{1}\times 2K_{1}\times K_{2}=3K_{1}\times P_{3}\times K_{1}$, applying Lemma \ref{lemma1}, we have $G=B\times K_3$.
If $G=(P_{2}+K_{1})\times 2K_{1}\times 2K_{1}$, by Lemma \ref{product laplacian}, the Laplacian spectrum of $G$ is
$\{[7]^2,6,[5]^2,4,0\}$, and then the Laplacian spectrum of $B$ is $\{4,[3]^2,2,0\}$. So $s(B)=72/5$, it is a contradiction.
\vspace{2mm}

\noindent\emph{Case 2.3.} $n=6$.   We have $e_0+e_1+e_2=0$. Then
\begin{eqnarray*}
G=4K_{1}\times 2K_{1}\times 2K_{1},  \text{~or~}  G=3K_1\times 3K_{1}\times 2K_{1}.
\end{eqnarray*}
The Laplacian spectrum of $G$ is $\{[8]^2,[6]^2,[4]^3,0\}$ or $\{[8]^2,6,[5]^4,0\}$. Then the Laplacian spectrum of $B$ is $\{[4]^2,[2]^3,0\}$ or $\{4,[3]^4,0\}$. Namely, the maximum Laplacian eigenvalue of $B$ is $4$. By Lemma \ref{min}, $\Delta(B)\leq 2$, a contradiction.
\vspace{2mm}\\
Hence $v_{m}=1$, \rm{i.e.,} $G_{m}=K_{1}$. Since $B\times
K_{m}=B\times K_{m-1}\times K_{1}$, by Lemma \ref{product
laplacian}, $B\times K_{m-1}$ is $L$-cospectral to $G_{0}\times
G_{1}\times \cdots \times G_{m-1}$. Using the induction hypothesis
on $m-1$, $B\times K_{m-1}$ and $G_{0}\times G_{1}\times \cdots
\times G_{m-1}$ are isomorphic. Hence, $B\times K_{m}= G$. This completes the proof.
\end{proof}

From \cite{kn:Wang10}, we know all graphs $\Theta_{r,1,s}$ with $r,s\neq 3$ are determined by their Laplacian spectra, then Theorem \ref{bicyclic} implies that $\Theta_{r,1,s}\times K_{m}$ with $r,s\neq 3$ are $L$-DS graphs.

\end{document}